\documentclass[oneside,notitlepage,12pt]{article}

\usepackage{amssymb}
\usepackage[leqno]{amsmath}
\usepackage{amsfonts}
\usepackage{latexsym}
\usepackage{amsopn}
\usepackage{amstext}
\usepackage{amsthm}

\usepackage[all]{xy}
\newdir{ >}{{}*!/-9pt/@{>}}

\usepackage{verbatim}
\usepackage[backref,colorlinks]{hyperref}

\providecommand{\cal}{\mathcal}
\renewcommand{\Bbb}{\mathbb}

\newcommand{\Aaa}{{\cal{A}}}
\newcommand{\Bee}{{\cal{B}}}

\newcommand{\Pee}{{\cal{P}}}
\newcommand{\Raa}{{\cal{R}}}
\newcommand{\See}{{\cal{S}}}

\newcommand{\Tee}{{\cal{T}}}

\newcommand{\Yu}{{\cal{U}}}

\newcommand{\Qyu}{{\Bbb{Q}}}

\renewcommand{\phi}{\varphi}
\renewcommand{\rho}{\varrho}
\newcommand{\RC}{\operatorname{RC}}
\newcommand{\RO}{\operatorname{RO}}
\newcommand{\CO}{\operatorname{CO}}

\newcommand{\rest}{\restriction}

\newcommand{\cl}{\operatorname{cl}}
\newcommand{\Int}{\operatorname{int}}

\renewcommand{\int}{\operatorname{int}}

\newcommand{\intb}{\int_{\beta X}}
\newcommand{\clb}{\cl_{\beta X}}

\newtheorem{tw}{Theorem}[section]
\newtheorem{wn}[tw]{Corollary}
\newtheorem{lm}[tw]{Lemma}
\newtheorem{prop}[tw]{Proposition}
\newtheorem{claim}[tw]{Claim}

\theoremstyle{definition}

\newtheorem{question}[tw]{Question}
\newtheorem{uw}[tw]{Remark}
\theoremstyle{remark}

\newcommand{\bR}{{\mathbb{R}}}

\newcommand{\bQ}{{\mathbb{Q}}}

\title{A generalization of $\varkappa$-metrizable  spaces }

\author{{\sc Andrzej Kucharski}
{\small Institute of Mathematics}\\
{\small University of Silesia}\\
{\small Katowice, POLAND}\\
{\small\texttt{akuchar@math.us.edu.pl}}
\and
{\sc S\l awomir Turek}
{\small Institute of Mathematics}\\
{\small Cardinal Stefan Wyszynski Univwersity in Warsaw}\\
{\small Warszawa, POLAND}\\
{\small\texttt{s.turek@uksw.edu.pl}}
}

\begin{document}

\maketitle

\begin{abstract}
 We introduce a new class of $\varkappa$-metrizable spaces, namely countably $\varkappa$-metrizable spaces. We show that the class of all $\varkappa$-metrizable spaces is a proper subclass of counably $\varkappa$-metrizable spaces. On the other hand,  for pseudocompact spaces the new class coincides with $\varkappa$-metrizable spaces. We prove a generalization of a Chigogidze result that the \v{C}ech-Stone compactification of a pseudocompact countably $\varkappa$-metrizable space is $\varkappa$-metrizable.
\end{abstract}
\date{}

\vspace{3mm}
\noindent
{\bf MSC(2010)}
Primary:
54B35; 
Secondary:
54D35, 

\noindent
{\bf Keywords:} \v{C}ech-Stone compactification, pseudocompact, open maps,  measurable cardinal, $\varkappa$-metrizable spaces.


\section{Introduction}

All topological spaces under consideration are assumed to be at least Tychonoff.

Recall that
a set $A\subseteq X$ is regular closed in a topological space $X$ if $\cl\Int A= A$.
For a topological space $X$ let $\RC(X)$ denote the set of all regular closed sets in $X$  and $\CO(X)$ denote the set of all closed and open sets in $X$. The family of all complements of sets from the family $\RC (X)$ forms family of regular open sets which will be denoted $\RO(X)$.
A topological space $X$ is $\varkappa$-\emph{metrizable} if  there exists  function $\rho:X\times\RC(X)\to[0,\infty)$ satisfying the following axioms
\begin{enumerate}
	\item [(K1)] $\rho(x,C)=0$ if and only if $x\in C$ for any $x\in X$ and $C\in \RC(X)$,
	\item [(K2)] If $C\subseteq D$, then $\rho(x,C)\geq \rho(x,D)$ for any $x\in X$ and $C,D\in\RC(X)$,
	\item [(K3)] $\rho(\cdot, C)$ is a continuous function,
	\item [(K4)] $\rho(x,\cl(\bigcup_{\alpha<\lambda} C_\alpha))=\inf_{\alpha<\lambda}\rho(x,C_\alpha)$ for any
	non-decreasing totally ordered sequence $\{C_\alpha:\alpha<\lambda\}\subseteq \RC(X)$ and any $x\in X$.
\end{enumerate}
We say that $\rho$ is $\varkappa$\textit{-metric} if it satisfies conditions $(K1)-(K4)$.
The concept of a $\varkappa$-metrizable space was introduced by Shchepin ~\cite{s76}.
The class of $\varkappa$-metrizable spaces is quite big; it contains (see e.g.~\cite{s76}, \cite{s79}, \cite{suz})

\begin{itemize}
\item  all metrizable spaces,

\item  Dugundji spaces,
\item all locally compact topological group,

\item the Sorgenfrey line.

\end{itemize}
Moreover
\begin{itemize}
\item a dense (open, regular closed) subspace of a $\varkappa$-metrizable space is
$\varkappa$-metrizable,
\item the product of any family of $\varkappa$-metrizable spaces is
$\varkappa$-metrizable,
\end{itemize}
On the other hand, the result of Chigogidze (see~\cite{ch82}) implies that compactifactions $\beta\omega$ and $\beta\mathbb{R}$ are not $\varkappa$-metrizable. Isiwata (see \cite{i86}) proved that the remainders $\beta\omega\setminus\omega$ and $\beta\mathbb{R}\setminus\mathbb{R}$ are not $\varkappa$-metrizable too.

\bigskip

 If $\rho\colon
 X\times\RC(X)\to[0,\infty)$ fulfills conditions  $(K1)-(K3)$ and condition
\[ (K4_\omega) \; \rho(x,\cl(\bigcup_{n<\omega} C_n))=\inf_{n<\omega}\rho(x,C_n)\mbox{ for any
	chain }\{C_n:n<\omega\}\mbox{ and any }x\in X \]
then we say that $\rho$ is \textit{countable $\varkappa$-metric} in $X$. A topological space which allows the existence of a countable $\varkappa$-metric we call {\em countably $\varkappa$-metrizable}.

\section{Representation of countably $\varkappa$-metrizable spaces}

We will show that every countably $\varkappa$-metrizable space which satisfies countable chain condition is $\varkappa$-metrizable. Countably $\varkappa$-metrizability differs from the notion of the $\varkappa$-metrizability under some set-theoretical assumption. Next we will prove that in the case of a pseudocompact space countably $\varkappa$-metrizabilty implies $\varkappa$-metrizability.

\begin{prop}\label{ccc}
If $X$ satisfies countable chain condition then every countable \mbox{$\varkappa$-metric} is a $\varkappa$-metric.
\end{prop}
\begin{proof}
Let  $\{C_\alpha:\alpha<\lambda\}\subseteq \RC(X)$ be a non-decreasing totally ordered sequence. By countable chain condition
there exists an increasing sequence $\{\alpha_n:n\in\omega\}\subseteq\lambda$ such that $\bigcup\{C_\alpha:\alpha<\lambda\}=\bigcup\{C_{\alpha_n}:n\in\omega\}$. Hence we get
$$\inf_{\alpha<\lambda}\rho(x,C_\alpha)\leq \inf_{n\in\omega}\rho(x,C_{\alpha_n})=\rho(x,\cl(\bigcup_{n\in\omega}C_{\alpha_n}))=
\rho(x,\cl(\bigcup_{\alpha<\lambda}C_\alpha))\leq \inf_{\alpha<\lambda}\rho(x,C_\alpha).$$

\end{proof}
We give an example of countably $\varkappa$-metrizable space  which is not $\varkappa$-metrizable space.

Let $\tau$ be an uncountable cardinal.
An ultrafilter $\Yu$ on $\tau$ is \textit{$\tau$-complete} if $\Aaa\subseteq \Yu$ and $|\Aaa|<\tau$,  then  $\bigcap\Aaa\in\Yu$.
An uncountable cardinal $\tau$ is\textit{ measurable} if there exists a  $\tau$-complete free ultrafilter $\Yu$ on $\tau$.

Let $\tau$ be an infinite cardinal and $\Yu$ be a free ultrafilter on $\tau$. Let $X=\tau\cup\{\Yu\}$ be space with  a topology inherited from \v{C}ech-Stone compactification of $\tau$ 

\begin{lm}\label{ult}
	If $C\in\RC(X)$, then $\Yu\in C$ 
	if and only if $C\cap\tau\in\Yu.$
\end{lm}
\begin{proof}
A neighborhood of the point $\Yu$ is of the form $D\cup\{\Yu\}$, where $D\in\Yu.$ If $\Yu\in C=\cl\int C$ then $ (\{\Yu\}\cup D)\cap \int C\ne\emptyset.$ Thus $D\cap C\cap\tau\ne\emptyset$ for all $D\in\Yu,$ and  by maximality of the fiter $\Yu$ we have $C\cap\tau\in\Yu.$
	If $C\cap\tau\in\Yu$ then obviously $\Yu\in \cl(C\cap\tau)\subseteq C$ 
\end{proof}

\begin{lm}\label{co}
	$\RC(X)=\CO(X)$
\end{lm}
\begin{proof}
	Obviously $\CO(X)\subseteq \RC(X)$. Let $C\in\RC(X)$ and consider the following cases.
	
	(1): $\Yu\notin C$. Then $C\subseteq \tau$ is an open subset of $X$.
	
	(2): $\Yu\in C$. Then  by Lemma~\ref{ult}$, C\cap\tau\in\Yu$. This finishes the proof because $C=\{\Yu\}\cup (C\cap\tau)$ is open set in $X$.
\end{proof}

\begin{tw} If $\tau$ is a measurable cardinal then the space $X=\tau\cup\{\Yu\}\subseteq\beta\tau$, where $\Yu$ is a $\tau$-complete free ultrafilter, is countably $\varkappa$-metrizable but not $\varkappa$-metrizable.
\end{tw}

\begin{proof}

Let $\varrho:X\times\RC(X)\to\{0,1\}$ be defined by
\[\rho(x,C)=\begin{cases}
1 &\text{ if } x\not\in C,\\
0 &\text{ if } x\in C.
\end{cases}\]

We claim that $\varrho$ is countable $\varkappa$-metric. Indeed, the function $\varrho$ satisfies conditions (K1) and (K2) of definitions of $\varkappa$-metric.
The function $f_C=\rho(\cdot,C)$ is continuous for any $C\in\RC(X).$ Let $U\subseteq \bR$ be an open set. If $U$ contains $0$ then $f^{-1}(U)=C$ and by Lemma~\ref{co} $C$ is open set. If $U$
contains $1$ then $f^{-1}(U)=X\setminus C$ and this is open set. So, $\varrho$ has the property (K3).
It remains to verify that the function has the property (K4${}_\omega$)
 Let $\{C_n:n\in\omega \}\subseteq \RC(X)$ be an increasing  sequence and $x\in X$. If there exists $n_0\in\omega$ such that $x\in C_{n_0}$ then we get
\[\rho(x,\cl \bigcup_{n\in\omega}C_n)=0=\rho(x, C_{n_0})=\inf\{\rho(x,C_n):n\in\omega\}.\]
Otherwise $x\notin C_n$ for every $n\in\omega$.
If $x\notin\cl\bigcup_{n\in\omega}C_n$ then obviously
$$\varrho(x, \cl\bigcup_{n\in\omega}C_n)=1=\inf\{\varrho(x,C_n)\colon n\in\omega\}.$$
Let us assume, therefore, that $x\in\cl\bigcup_{n\in\omega}C_n$.
Then  $x=\Yu$ and by Lemma~\ref{ult},
$\tau\setminus C_n\in\Yu$
for every $n\in\omega$. Since $\Yu$ is $\sigma$-complete ultrafilter we get $D=\bigcap_{n\in\omega}(\tau\setminus C_n)\in\Yu.$ Hence
\[(D\cup\{\Yu\})\cap\bigcup_{n\in\omega}C_n=\emptyset,\]
a contradiction.

We shall prove that the space $X$ is not   $\varkappa$-metrizable.

  Let $C_\alpha=\alpha$ for
each $\alpha<\tau$. Since $\alpha\notin\Yu$  the set $C_\alpha$ is clopen. Suppose that $\rho'$ is a $\varkappa$-metric on $X$. Since $\Yu\in \cl\bigcup\{C_\alpha:\alpha<\tau\}$ we
get $\rho'(\Yu,\cl\bigcup\{C_\alpha\colon\alpha<\tau\})=0$. Therefore $\inf\{\rho'(\Yu,C_\alpha)\colon\alpha<\tau\}=0$ and there exists an increasing sequence $\{\alpha_n:n\in\omega\}$ such that $\inf\{\rho'(\Yu,C_{\alpha_n})\colon n\in\omega\}=0$. Let $\alpha=\sup\{\alpha_n:n\in\omega \}$ or in other words $C_\alpha=\bigcup\{C_{\alpha_n}\colon n\in\omega\}$. Hence $ 0=\inf\{\rho'(\Yu,C_{\alpha_n})\colon n\in\omega\}=\rho'(\Yu,C_\alpha)>0,$ a contradiction.
\end{proof}

A countable $\varkappa$-metric  $\varrho\colon X\times\RC(X)\to\{0,1\}$  we will call {\em two-valued}.

\begin{prop}
Assume that  $\tau$ is an infinite cardinal and $\Yu$ is a free ultrafilter on $\tau$ and $X=\tau\cup\{\Yu\}\subset\beta\tau.$ If  $\rho\colon X\times\RC(X)\to\{0,1\}$ is a two-valued countable $\varkappa$-metric then
$\Yu$ is $\aleph_1$-complete on $\tau$.\end{prop}
\begin{proof}
Suppose that there exists $\{D_n\colon n\in\omega\}\subseteq\Yu$ and $\bigcap\{D_n:n\in\omega\}\notin\Yu$. We can assume that $D_{n+1}\subseteq D_n$ for all $n\in\omega.$ Let $E_n=\cl D_n$. Then $E_n\in\CO(X)$ and $\Yu\in E_n$. Since $\rho(\Yu,X\setminus E_n)=1$ for all $n\in\omega$  and  $\cl\bigcup\{X\setminus E_n\colon n\in\omega\}=\cl (X\setminus \bigcap\{E_n\colon n\in\omega\})=\cl (\tau\setminus\bigcap\{D_n\colon n\in\omega\})$  we get $$0=\rho(\Yu,\cl\bigcup\{X\setminus E_n:n\in\omega \})=\inf\{\rho(\Yu,X\setminus E_n):n\in\omega\}=1,$$ a contradiction.
\end{proof}

\begin{uw}
Assume  $\tau$ is the least cardinal that carries a two-valued countable
$\varkappa$-metric on $\tau\cup\{\Yu\}$. Hence
$\tau$ is the least $\aleph_1$-complete cardinal. By
\cite[Lemma 10.2]{jech} $\tau$ is measurable cardinal.
\end{uw}

Since measurable cardinals are large cardinals whose existence cannot be proved from ZFC, it is natural to ask the question:

\begin{question}
	Does there exist in ZFC
a countably $\varkappa$-metrizable space which is not $\varkappa$-metrizable?
\end{question}
Assume that $X$ is a pseudocompact space. Now we show that each $\varkappa$-metrizable pseudocompact space has
a special representation as an inverse  limit. In order to obtain this representation we use some ideas from article~\cite{kp8} and monograph~\cite{hsh}.

 A continuous surjection $f:X \to Y$ is said to be \textit{d-open} if
 $f[U]\subseteq\int\cl f[U]$ for any open set $U\subseteq X$.
The notion of d-open maps was introduced by Tkachenko \cite{tk}. A function
$$f:\bR\times\{0\}\cup\bQ\times\{1\}\to\bR,$$
defined in the following way  $f(x,0)=x$ for any $x\in\bR$ and $f(x,1)=x$ for any $x\in\bQ$ is an example of d-open but not open map. We will use the following  Proposition (see  \cite{tk} or  \cite{kpv}).

\begin{prop}\label{d-open}
 Let $f\colon X \to Y$ be a continuous function, then the following condition are equivalent:
 \begin{enumerate}
    \item $f$ is a d-open map,
	\item there exists a base $\Bee_Y\subseteq \Tee_Y$ such that $\Pee=\{f^{-1}(V): V \in \Bee_Y\} \subseteq_! \Tee_X$,
	i.e. for any $\See\subset \Pee$ and $x\notin \cl_X \bigcup\See$, there exists $W\in\Pee$ such that $x\in W$
and $W\cap\bigcup \See=\emptyset$.
\qed
\end{enumerate}
\end{prop}

\begin{lm}\label{claim7}
	Let $X$ be pseudocompact and $Y$ be a second  countable regular space and let $f:X\to Y$ be a d-open map. Then $f[\cl V]=\cl f[V]$ for any open subset $V\subseteq X$ and $f$ is open map
\end{lm}

\begin{proof}
	Let $V\subseteq X$ be an open nonempty set.
	 It is known (see e.g.~\cite[Ex.~3.10.F(d)]{eng}) that $\cl W$ is pseudocompact for any open nonempty set $W\subseteq X$. Since $Y$ is separable metric space and continuous image of pseudocompact space is compact,
	so the image $f[\cl W]$ is compact subspace for any open subset $W\subseteq X$. Therefore $\cl f[V]=f[\cl V]$ for any open set $V\subseteq X$.
	It remains to prove that
$f$ is open map. To this end, consider an open set $U\subseteq X$ and  $x\in U$. There exists an open neighbourhood $V$ of $x$ such that $x\in V\subseteq\cl V\subseteq U$. Then
\[f(x)\in f[V]\subseteq \int\cl f[V]=\int f[\cl V]\subseteq f[\cl V]\subseteq f[U],\]
this completes the proof.
\end{proof}

Let $\Pee$ be a family of subsets of $X$.
Let define an equivalent relations on $X$. We say that $x\sim_\Pee y$ if and only if $$x\in V\leftrightarrow  y\in V
\mbox{ for  every } V\in\Pee.$$
Denote  by $[x]_{\Pee} $ the class of elements which is equivalent to $x$ with respect to $\sim_\Pee $.
By $X_\Pee$ we will denote a set $\{[x]_\Pee\colon x\in X\}$ and by $q\colon X\to X_\Pee$ a map $q(x)=[x]_\Pee$.
It is clear that $q^{-1}(q(V))=V$ for each $V\in\Pee$.

Let $X$ be a countably $\varkappa$-metrizable space with  countable $\varkappa$-metric $\varrho$.
We say that $\Pee\subseteq \RO(X) $
is \textit{$\Qyu$-admissible} if it satisfies the following conditions:
\begin{itemize}
\item if $V\in\Pee$ then $\Int f^{-1}_{\cl V}((-\infty,q]),
\Int f^{-1}_{\cl V}([q,\infty))\in\Pee$ for all $q\in\Qyu$, where $f_{\cl V}(\cdot )=\rho(\cdot ,\cl V)$,
\item $\Int(X\setminus V)\in\Pee$  for all $V\in\Pee$,
\item $ U\cap V\in\Pee$ and $\int\cl(U\cup V)\in\Pee$ for all $U,V\in\Pee$.
\end{itemize}

Applying inductive argument we can prove the following fact.

\begin{lm}\label{dir}
	For any family $\Aaa\subseteq \RO(X)$ there is a
	$\Qyu$-admissible family $\Pee$ such that $\Aaa\subseteq \Pee$ and $|\Pee|\le \aleph_0\cdot |\Aaa|$.\qed
\end{lm}

\begin{lm}\label{c1}
	If $\Pee$ is a  $\Qyu$-admissible family, then for each $V\in\Pee$ there exists a sequence of increasing regular open sets $\{V_n: n\in\omega\}\subseteq\Pee$ such that
	\begin{equation*}
	V=\bigcup_{n\in\omega} V_n \text{ and }V_n\subseteq \cl V_n\subseteq V_{n+1}\subseteq V\text{ for each }n\in\omega.\tag{$*$}
	\end{equation*}
\end{lm}

\begin{proof}
	Define
	$V_n=\int f^{-1}([\frac {1}{ n+1},\infty])$ where $f(x)=\rho(x, X\setminus V)$.
	Since $X\setminus \cl V=\int(X\setminus V)\in \Pee$ then $V_n\in\Pee$.
	It is easy to see that
	$V=\bigcup\limits_{n\in\omega} V_n$ and $V_n\subseteq \cl V_n\subseteq V_{n+1}\subseteq V$.
\end{proof}

\begin{lm}\label{c2}
 If $\Pee$ is a $\Qyu$-admissible family , then
 $[x]_{\Pee} =\bigcap\{\cl V: x\in V\in\Pee\}$.
\end{lm}
\begin{proof}
Obviously $[x]_{\Pee}\subseteq\bigcap\{\cl V\colon x\in V\in\Pee\}$.
Let
 $a\in \bigcap\{\cl V\colon x\in V\in\Pee\}$. Suppose that there exists $V\in\Pee$ such that $x\in V$ and $a\not\in V$.
By lemma~\ref{c1} we have $\{V_n\colon n\in\omega\}\subseteq\Pee$ that satisfies $(*)$. There exists $n\in\omega$ such that
$x\in V_n$. But  $a\in \cl V_n\subseteq V_{n+1}\subseteq V$; a contradiction.
Now suppose that $a\in W$ and $x\notin W$, where $W\in\Pee$.
Let $\{W_n\colon n\in\omega\}\subseteq\Pee$ satisfies $(*)$ for $W$.
There is $n\in\omega$ such that $a\in W_n$. On the other hand
 $x\in X\setminus \cl W_n=\int (X\setminus W_n)\in\Pee$,
and $a\in \cl\int(X\setminus W_n)=
X\setminus W_n$, a contradiction with $a\in  W_n$.
\end{proof}

\begin{lm}\label{c3}
 If $\Pee$ is a countable $\Qyu$-admissible  and $x\sim_\Pee y$, then
 $\rho(x,\cl\bigcup\Aaa)=\rho(y,\cl\bigcup\Aaa)$ for any $\Aaa\subseteq \Pee$.
\end{lm}
\begin{proof}Let $x,y\in X$ and $\Aaa\subseteq \Pee$.

Assume first that $\Aaa\subseteq \Pee$ is finite. Suppose that $\rho(x,\cl\bigcup\Aaa)>\rho(y,\cl\bigcup\Aaa)$.
There exists $q\in\Qyu$ such that $\rho(x,\cl\bigcup\Aaa)\ge q>\rho(y,\cl\bigcup\Aaa)$. Since $\Pee$ is $\Qyu$-admissible,
$V=\int\cl\bigcup\Aaa\in\Pee$ and $\cl V=\cl\bigcup\Aaa$. For the map $f_{\cl V}(\cdot)=\rho(\cdot,\cl V)$, we get
$x\in\int f^{-1}_{\cl V}([q,\infty))\in\Pee$ and $y\not\in\int f^{-1}_{\cl V}([q,\infty))$, a contradiction with $x\sim_\Pee y$.

Assume now that $\Aaa$ is countable and  infinite. We decompose $\Aaa$ into the sum of a strictly increasing sequence of families $\Aaa_n\subseteq\Aaa$
of strictly increasing cardinalities, $\Aaa=\bigcup\{\Aaa_n:n\in\omega\}$. Then
$$\rho(x,\cl\bigcup\Aaa)=\inf\{
\rho(x,\cl\bigcup\Aaa_n)\colon n\in\omega\}=\inf\{\rho(y,\cl\bigcup\Aaa_n)\colon n\in\omega\}=\rho(y,\cl\bigcup\Aaa)$$ by condition $(K4_\omega)$.
\end{proof}

Let $X$ be a a pseudocompact countably $\varkappa$-metrizable space and $\Pee$ be a $\Qyu$-admissible family. 
The set $X_\Pee=\{[x]_\Pee\colon x\in X\}$ is equipped with the topology $\Tee_\Pee$ generated by all images $q[V], V \in \Pee.$ Since $\Pee$ is $\Qyu$-admissible it is closed under finite intersection and $X =\bigcup\Pee.$

\begin{lm}[{\cite[Lemma 1]{kp8}}]\label{l1}
The mapping $q : X\to X_\Pee$ is continuous provided $\Pee$ is an open family $X$ which is closed under finite intersection. Moreover, if $X =\bigcup\Pee,$ then the family $\{ q [V] : V \in\Pee\}$ is a base for the topology $\Tee_\Pee$.
\end{lm}

 To show that if $X$  is pseudocompact countably $\varkappa$-metrizable space, then  $X_\Pee$ is Tychonoff space, we apply the following Frink's theorem, see \cite{fri}. 

\bigskip

\textbf{Theorem} [O. Frink (1964)]. \textit{ A $T_1$-space $X$ is Tychonoff if and only if there exists  a base $\Bee$ satisfying}:\\
\indent (1) \textit{If  $x\in U\in \Bee$, then there exists  $V\in \Bee$ such that $x\not\in V$ and $U\cup V = X$};\\
\indent (2) \textit{ If   $U,V\in \Bee$ and
 $U\cup V= X$, then there exists
disjoint sets $M,N \in \Bee$ such that $X\setminus U\subseteq M$ and   $X\setminus V\subseteq N$}.  \qed

\bigskip


\begin{tw}\label{t5} Let $X$ be a pseudocompact countably $\varkappa$-metrizable space and $\Pee $ be countable  $\Qyu$-admissible  family. Then the $X_{\Pee}$ is  compact and metrizable.
\end{tw}

\begin{proof}
First we shall prove that $X_{\Pee}$ is $T_1$-space. Let $[x]_\Pee\ne[y]_\Pee$. Then there exists $V\in\Pee$ such that
$x\in V$ and $y\not\in V$.
By virtue of Claim~\ref{c1} there is a family $\{V_n:n\in\omega\}\subseteq\Pee$ which satisfies condition $(*)$ for set $V$.
So there is $n\in\omega$ such that $x\in V_n$.
Since
$y\in X\setminus \cl V_n\in\Pee$ then $[y]_\Pee\in q[X\setminus \cl V_n]$ but $[x]_\Pee\notin q[X\setminus\cl V_n]$.

We shall prove that $X_\Pee$ satisfies  condition $(1)$ of Frink's theorem with a base $\Bee=\{q[V]:V\in \Pee\}$.
Fix $[x]_\Pee\in U=q[V]$ where $V\in\Pee$. Since $V=\bigcup\{V_n:n\in\omega \}$, there exists $V_n\in\Pee$ such that
$x\in V_n\subseteq\cl V_n\subseteq V$. Therefore $x\not\in X\setminus \cl V_n\in\Pee$ and $V\cup (X\setminus \cl V_n)=X$. So, $X_\Pee=q[V]\cup q[X\setminus V_n]=U\cup q[X\setminus V_n]$.

 Let prove condition $(2)$. Fix $U, V\in\Pee$ such that $U\cup V=X$. By Claim~ \ref{c1} there are $\{ V_n: n\in\omega\}\subseteq \Pee$ and $\{U_n:n\in\omega \}\subseteq \Pee$ such that  $V=\bigcup\{V_n:n\in\omega \}$ and
  $U=\bigcup\{U_n:n\in\omega \}$ and $V_n\subseteq \cl V_{n+1}$
 and $U_n\subseteq \cl U_{n+1}$. Since $(X\setminus U)\cap(X\setminus V)=\emptyset$ and $X$ is pseudocompact there exists
 $n\in\omega $ such that $(X\setminus \cl U_n)\cap(X\setminus \cl V_n)=\emptyset$ and
 $X\setminus U\subseteq X\setminus \cl U_n\in\Pee$ and $X\setminus V\subseteq X\setminus \cl V_n\in\Pee$

By the Frink's theorem $X$ is Tychonoff. So, $X$ is metrizable by Urysohn Metrization theorem. Since $X_\Pee$ is continuous image of pseudocompact space $X$ then $X$ is compact.
 \end{proof}

\begin{lm}\label{factor}
 Let $X$ be a countably $\varkappa$-metrizable pseudocompact space and $\Aaa$ be a countable family of regular open sets.
 There exists a d-open map $f\colon X\to Y$ onto a compact metrizable space with a countable base $\Bee$ such that $\Aaa\subseteq
  f^{-1}(\Bee)$

\end{lm}

\begin{proof}
Let $\Aaa$ be a countable family of regular open sets  and let $\Pee\subseteq \RO(X)$ be a  countable $\Qyu$-admissible family  such that $\Aaa\cup\{X\}\subseteq \Pee$. Let $Y=X_\Pee$ and $f=q\colon X\to X_\Pee$.

By Lemma \ref{l1} the function $f$ is continuous and
$\Aaa\subseteq f^{-1}(\Bee)$ for the base $\Bee=\{f[V]\colon V\in\Pee\}$
of $Y$.
By
Lemma~\ref{t5}
the space $Y$ is compact and metrizable.

 It remains to show that $f$ is d-open. By
 Proposition \ref{d-open}
it is enough to show that for any  $\See\subseteq \Pee$ and any $x\not\in\cl\bigcup\See$ there is $V\in\Pee$ such that $x\in V$ and $V\cap \bigcup\See=\emptyset$.
To do  this fix arbitrary $\See\subseteq \Pee$ and $x\not\in\cl \bigcup\See$. We get $\rho(x,\cl \bigcup\See)>0$,
where $\rho$ is countable $\varkappa$-metric on $X$.
Hence by Lemma~\ref{c3}
$[x]_{\Pee}\cap \cl \bigcup\See=\emptyset$. In other words, by Lemma~\ref{c2}, we have $\bigcap\{\cl V:x\in V\in\Pee\}\cap \cl \bigcup\See=\emptyset$. Since $X$ is pseudocopmact there is $V_1,\ldots ,V_n\in\Pee$ such that $V_1\cap\ldots\cap V_n\cap \bigcup\See=\emptyset$ and $x\in V_1\cap\ldots\cap V_n$. Let $V=V_1\cap\ldots\cap V_n$. The set $V$ has requiered properties.
\end{proof}

The notion of an almost limit  was introduced by Valov \cite{val}.
We say that a space $X$ is \textit{an almost limit} of the inverse
system $\displaystyle S=\{X_\sigma, \pi^{\sigma}_\varrho, \Sigma\}$,
if there exists an embeding $q:X\to\displaystyle\underleftarrow{\lim}S$ such
that $\pi_\sigma[q[X]]=X_\sigma$ for each $\sigma\in\Sigma$. We denote
this by $X=\displaystyle\mathrm{a}-\underleftarrow{\lim}S$. Obviously, if $X=\displaystyle\mathrm{a}-\underleftarrow{\lim}S$
then $X$ is a dense subset of
$\displaystyle\underleftarrow{\lim} S$.

\begin{tw}\label{open-rep}

If $X$ is pseudocompact countably $\varkappa$-metrizable space, then  $$ X = a-\varprojlim \{ X_\sigma, \pi^\sigma_\varrho,\Sigma\},$$ where   $\{ X_\sigma, \pi^\sigma_\varrho,\Sigma\}$ is a  $\sigma$-complete inverse system, all spaces $X_\sigma$ are compact and metrizable with countable weight, and all bonding maps $\pi^\sigma_\varrho$ are open. Moreover the space
$Y=\varprojlim \{ X_\sigma, \pi^\sigma_\varrho,\Sigma\}$ is \v{C}ech-Stone compactification of $X$.
\end{tw}

\begin{proof}
	Let $\Sigma=\{\Pee\subseteq\RO(X)\colon \Pee\mbox{  is a countable }\Qyu\mbox{-admissible family}\}$. The family $\Sigma$ ordered by inclusion is directed by Lemma~\ref{dir}.  Let consider $q_\Pee\colon X\to X_\Pee$ for each $\Pee\in\Sigma$. By Lemma~\ref{l1} $q_\Pee$ is continuous map and by Lemma~\ref{t5}, each $X_\Pee$ is a compact metrizable space.
	If $\Pee\subseteq\Raa$, where $\Pee,\Raa\in\Sigma$, then we have naturally defined map $\pi_{\Pee}^{\Raa}\colon X_{\Raa}\to X_{\Pee}$ such that the diagram
	$$\xymatrix{
		& X\ar[ld]_{q_{\Pee}}\ar[rd]^{q_{\Raa}}&\\
		X_{\Pee}& & X_{\Raa}\ar[ll]^{\pi_{\Pee}^{\Raa}}
	}	
	$$
	commutes. It is quite obvious that $\pi_\Pee^\Raa$ is open. Moreover for each increasing chain $\{\Pee_n\colon n\in\omega\}$ in $\Sigma$ the space $X_{\Pee}$, where $\Pee=\bigcup\{\Pee_n\colon n\in\omega\}$, is homeomorphic to $\varprojlim\{X_{\Pee_n},\omega \}$.	So,
	$\{X_\Pee,\pi_\Pee^\Raa,\Sigma\}$ constitutes a $\sigma$-complete inverse system, where all spaces $X_\Pee$ are compact and metrizable and all bonding maps $\pi_\Pee^\Raa$ are open.
	
	Now we check that the limit map
	$q=\varprojlim(q_\Pee,\Sigma)\colon X\to \varprojlim\{X_\Pee,q_\Pee^\Raa,\Sigma\}$ is an embedding. If $x,y\in X$ and $x\ne y$ then we can find a disjoint regular open sets $U,V$ such that $x\in U$ and $y\in V$. By Lemma~\ref{dir} there is $\Pee\in\Sigma$ such that $U,V\in\Pee$. So, $q(x)\ne q(y)$ because $[x]_\Pee\ne[y]_\Pee$. It remains to prove that $q[V]$ is open in $q[X]$ whenever $V\subseteq X$ is open.  We can assume that $V\in\RO(X)$. Let $\Pee\in\Sigma$ be such a family that $V\in\Pee$.
	Note that $q[V]=q[X]\cap\pi^{-1}_\Pee(q_\Pee[V])$, where $\pi_\Pee$ denotes projection from the inverse limit to $X_\Pee$. Let $a\in q[X]\cap \pi^{-1}_\Pee[q_\Pee[V]]$. Then $a=q(x)$ for some $x\in X$ and hence
	$$q_\Pee(x)=q(x)_\Pee=\pi_\Pee(q(x))=\pi_\Pee(a)\in q_\Pee[V].$$
	It means that $[x]_\Pee=[y]_\Pee$ for some $y\in V$. Therefore $x\in V$ and $a=q(x)\in q[V]$. An inclusion $q[V]\subseteq q[X]\cap\pi^{-1}_\Pee(q_\Pee[V])$
	is obvious.

\begin{claim}
For any continuous map $h:X\to[0,1]$ there exists $\Pee\in\Sigma$  and a contiunous map $g\colon X_\Pee\to[0,1]$ such that
$h=g\circ\pi_\Pee\rest X$.
\end{claim}	
	Let $h:X\to[0,1]$ and $a\in[0,1]$. There exists a sequence of open subsets $\{U^a_n:n\in\omega\}$ such that
$U^a_{n+1}\subseteq\cl U^a_{n+1}\subseteq U^a_n\subseteq[0,1]$ and $\{a\}=\bigcap_{n\in\omega} U^a_n$. Since $[0,1]$ has a
countable base $\Bee$ we may assume that $\{U^a_n:n\in\omega\}\subseteq \Bee.$  Therefore we get
\[h^{-1}(U^a_{n+1})\subseteq \int\cl h^{-1}(U^a_{n+1})\subseteq \cl h^{-1}( U^a_{n+1})\subseteq h^{-1}( \cl U^a_{n+1})\subseteq h^{-1}(U^a_n)\]	and
\[h^{-1}(\{a\})=\bigcap_{n\in\omega}\int\cl h^{-1}(U^a_n).\] There exists $\Pee\in\Sigma$ such that $\{\int\cl h^{-1}(U^a_n)\colon n\in\omega,\; a\in [0,1]\}\subseteq \Pee$. Now we shall prove that $h(x_1)=h(x_2)$, whenever $x_1\sim_{\Pee}x_2$. Let $y=h(x_1)$. Then $x_1\in \bigcap_{n\in\omega}\int\cl h^{-1}(U^y_n)=h^{-1}(y)$ and $\{\int\cl h^{-1}(U^y_n)\colon n\in\omega\}\subseteq\Pee$. Since $x_1\sim_{\Pee}x_2$, we get  $x_2\in \int\cl h^{-1}(U^y_n)$ for every $n\in\omega$, what implies  $x_2\in h^{-1}(y).$ Define a map $g$ by the formula $g([x]_{\Pee})=h(x)$ for any $x\in X$. In order to show that $g$ is continuous we will prove that  $g^{-1}(U)=\pi_\Pee(h^{-1}(U))$ for any open subset $U\subseteq[0,1].$  We have the following equivalence
\[[x]_\Pee\in g^{-1}(U)\Leftrightarrow g([x]_\Pee)\in U\Leftrightarrow h(x)\in U\Leftrightarrow x\in h^{-1}(U) .\]
Since $h(x_1)=h(x_2)$ whenever $x_1\sim_{\Pee}x_2$, we get $x\in h^{-1}(U)\Leftrightarrow [x]_\Pee\in \pi_\Pee(h^{-1}(U)).$ This competes the proof of the claim.

Given a continuous map $h:X\to [0,1]$ by the Claim there  exists $\Pee\in\Sigma$  and a contiunous map $g:X_\Pee\to[0,1]$ such that $h=g\circ \pi_\Pee\rest X$. The map $g\circ \pi_\Pee:Y\to[0,1]$ is required extension of $h$, hence $Y=\beta X.$
\end{proof}

By the result of Kucharski~\cite{k12} (see also \cite{bla}) and Theorem \ref{open-rep} we get the following
Corollary, which generalizes Chigogidze's result~\cite[Corollary 2]{ch82} about pseudocompact $\varkappa$-metrizable space.

\begin{wn}\label{kappa-ccc}
Any pseudocompact conutably $\varkappa$-metrizable space is ccc.
\end{wn}

For example
$\omega_1$ with the order topology is  not countably $\varkappa$-metrizable since it is pseudocompact and is not ccc.

We have proved that for ccc spaces or  pseudocompact spaces, conutably $\varkappa$-metrizable spaces coincides with $\varkappa$-metrizable spaces.

\begin{question}
For which class of spaces does conutably $\varkappa$-metrizability  coincide with $\varkappa$-metrizability?
\end{question}

\section{ \v{C}ech-Stone compactification of $\varkappa$-metrizable\\ space }

A. Chigogidze announced in \cite{ch82} without a proof that \v{C}ech-Stone compactification of $\varkappa$-metrizable space is $\varkappa$-metrizable. Next G. Dimov gave in \cite{d83} sufficient and necessary conditions for compact Hausdorff extension of $\varkappa$-metrizable space to be $\varkappa$-metrizable. In this section we will give a simple  proof that \v{C}ech-Stone compactification of pseudocompact countably $\varkappa$-metrizable space is $\varkappa$-metrizable.

Firstly, note the following simple observation.

\begin{lm}\label{claim6}
	Let $X$ be completely regular space. If $F\in\RC(\beta X)$ then $F\cap X\in\RC(X)$. \qed
\end{lm}

All topological spaces $X$ considered below are assumed to be pseudocompact and
countably $\varkappa$-metrizable, with countable $\varkappa$-metric $\rho$. So, each continuous function $\rho(\cdot ,F\cap X)\colon  X\to[0,\infty)$ is bounded say by $b_F\in[0,\infty)$ for each regular closed subset $F\subseteq\beta X$. Hence we can extend each function $ \rho(\cdot ,F\cap X)$ to continuous function $\bar{\rho}(\cdot ,F\cap X):\beta X\to
[0,b_F]$. Now we shall prove that a
function $\psi\colon \beta X\times \RC(\beta X)\to \mathbb{R}$ defined by the formula
$$\psi(p, F)=\bar{\varrho}(p, F\cap X)$$
is countable $\varkappa$-metric or
satisfies condition $(K1)$, $(K2)$ and $(K4_\omega)$ (the condition $(K3)$ is obviously fulfilled).

\begin{claim}[$K1$]\label{K1}
	Let $F'\in\RC(\beta X)$ and  $F=F'\cap X$.
	\begin{enumerate}
		\item If $p\in F'$ then $\bar{\rho}(p ,F)=0$.
		\item if $\bar{\rho}(p ,F)=0$, then $p\in F'$.
		
	\end{enumerate}
\end{claim}

\begin{proof} $(1)$
	Suppose that $\bar{\rho}(p ,F)>0$. Then there is an open neighbourhood $V\subseteq \beta X$ of $p$ and $b>0$ such that
	$\bar{\rho}(q ,F)>b$ for each $q\in V$. Since $V\cap \intb F'\ne\emptyset$ there is $y\in X\cap V\cap \intb F'$. Therefore
	$b<\bar{\rho}(y ,F)=\rho(y ,F)=0$, a contradiction.
	
	$(2)$  Suppose that $p\not\in F'$. There are open neighborhoods $V_n\subseteq \beta X$ of $p$ such that $\bar{\rho}(y,F)<\frac 1 n$ for $y\in V_n$, $V_n\cap F'=\emptyset$ and $\clb V_{n+1}\subseteq V_n$. Since $X$ is pseudocompact space
	$\bigcap_{n\in\omega}\clb V_n\cap X\ne\emptyset$. Let $y_0\in \bigcap_{n\in\omega}\clb V_n\cap X$. Therefore $\frac 1 n
	\geq\bar{\rho}(y_0,F)=\rho(y_0,F)>0$ for each $n\in\omega$, a contradiction.
\end{proof}

\begin{claim}[$K2$]\label{K2}
	Let $F',G'\in\RC(\beta X)$ and  $F=F'\cap X$ and $G=G'\cap X$. If $F'\subseteq G',$ then $\bar{\rho}(p ,F)\geq \bar{\rho}(p ,G)$
	for every $p\in\beta X$.
\end{claim}

\begin{proof}
	Suppose that there exists $p\in\beta X$ such that $ \bar{\rho}(p ,F)< \bar{\rho}(p ,G).$ There is a neighborhood $V$ of
	$p$ such that $ \bar{\rho}(q ,F)< \bar{\rho}(q ,G)$ for every $q\in V$. Let $a\in V_x\cap X$, then $\rho(a ,F)= \bar{\rho}(a ,F)<\bar{\rho}(a,G)=\rho(a ,G)$, a contradiction.
\end{proof}

\begin{claim}[$K4_\omega$]\label{K4}
	Let $\{F_n'\colon n\in\omega\}\subseteq\RC(\beta X)$ be a such family that $F'_n\subseteq F'_{n+1}$ and let $F_n=F_n'\cap X$. Then
	$\bar{\rho}(p ,\clb(\bigcup_{n\in\omega}F_n)\cap X)=\inf\{\bar{\rho}(p ,F_n):n\in\omega\}$
	for every $p\in\beta X$.
\end{claim}

\begin{proof} Let $F=\clb(\bigcup_{n\in\omega}F_n)\cap X$. By Claim \ref{K2} we get an inequality $"\leq "$.
	Suppose that $\bar{\rho}(x ,F)<b<\inf\{\bar{\rho}(x ,F_n):n\in\omega\}=a.$ There exists neighbourhood $V$
	of $x$ such that $\bar{\rho}(y ,F)<b$ for every $y\in V$. There is a sequence $\{V_n:n\in\omega\}$ of open
	sets of $\beta X$ such that $\clb V_{n+1}\subseteq V_n\subseteq\clb V_n\subseteq V$ and $\bar{\rho}(y ,F_n)\in
	(a-\frac 1 n,a+\frac 1 n)$ for every $y\in V_n$. Since $X$ is pseudocompact there is $y\in\bigcap_{n\in\omega}
	\clb V_n\cap X$. Therefore
		$$b<a=\inf\{\bar{\rho}(y,F_n):n\in\omega\}=\inf\{\rho(y ,F_n):n\in\omega\}=\rho(y ,F)=\bar{\rho}(y,F)<b,$$
		a contradiction.
\end{proof}
\begin{tw}
	If $X$ is pseudocompact countably $\varkappa$-metrizable space then $\beta X$ is $\varkappa$-metrizable.
\end{tw}

\begin{proof}
	We use Remark~\ref{ccc}, Corollary~\ref{kappa-ccc} and previous Claims.
\end{proof}

\end{document}